\documentclass[12pt]{article}

\usepackage[margin=1in]{geometry}

\usepackage{float}

\usepackage{relsize}

\usepackage{diagbox}
\usepackage{mathtools}
\usepackage{bbm}
\usepackage{latexsym}
\usepackage{epsfig}
\usepackage{amsmath,amsthm,amssymb,enumerate}

\allowdisplaybreaks
\def\bZ{{\bf Z}}
\def\SM{\text{SMALL}}
\usepackage{color}
\usepackage[usenames,dvipsnames]{xcolor}

\def\cH{{\mathcal H}}
\def\nn{\nonumber}
\def\a{\alpha}  \def\d{\delta} \def\D{\Delta}
\def\e{\varepsilon} \def\f{\phi}   
\def\G{\Gamma}  \def\k{\kappa}
\def\z{\zeta} \def\th{\theta}    
 \def\m{\mu} \def\n{\nu} \def\p{\pi}
\def\r{\rho}  \def\s{\sigma} 
\def\t{\tau} \def\om{\omega}  \def\Om{\Omega}\def\U{\Upsilon}

\renewcommand{\c}[1]{{\mathcal #1}}
\def\cP{{\cal P}}

\def\bx{{\bf x}}

\newtheorem{theorem}{Theorem}
\newtheorem{lemma}[theorem]{Lemma}

\newcommand{\wh}[1]{\widehat{#1}}

\newcommand{\rdup}[1]{{\left\lceil #1\right\rceil }}
\newcommand{\rdown}[1]{{\left\lfloor #1\right \rfloor}}

\newcommand{\brac}[1]{\left(#1\right)}

\newcommand{\bfrac}[2]{\left(\frac{#1}{#2}\right)}

\def\cE{{\cal E}}
\newcommand{\rai}{\rightarrow \infty}

\newcommand{\set}[1]{\left\{#1\right\}}
\def\sm{\setminus}
\def\seq{\subseteq}
\def\es{\emptyset}

\def\E{\mathbb{E}}
\def\Var{\mathbb{V}ar}
\def\Pr{\mathbb{P}}

\def\cF{{\cal F}}
\newcommand{\ignore}[1]{}

\def\cB{{\mathcal B}}
\def\cC{{\mathcal C}}
\def\cD{{\mathcal D}}
\def\cE{{\mathcal E}}
\def\cF{{\mathcal F}}

\def\cH{{\mathcal H}}

\def\cM{{\mathcal M}}

\def\cP{{\mathcal P}}

\newcommand{\card}[1]{\left|#1\right|}

\newcommand{\beq}[2]{\begin{equation}\label{#1}#2\end{equation}}

\def\nn{\nonumber}

\def\bd{{\bf d}}

\def\bz{{\bf z}}

\newcommand{\dnm}[1]{D_{n,m}^{(\d\geq#1)}}
\newcommand{\cdnm}[1]{{\mathcal D}_{n,m}^{(\delta\geq#1)}}
\newcommand{\gnm}[1]{G_{n,m}^{(\d\geq#1)}}
\def\seq{\subseteq}

\usepackage{tikz}
\tikzset{ b/.style = { circle, draw , thick, inner sep = 0pt
                    , fill = black
                   , minimum size = 2.5pt                     }}

\begin{document}
\author{Colin Cooper\thanks{Research supported  at the University of Hamburg, by a  Mercator fellowship from DFG Project 491453517}\\Department of Informatics\\
King's College\\
London WC2B 4BG\\England
\and
Alan Frieze\thanks{Research supported in part by NSF grant DMS1952285}\\Department of Mathematical Sciences\\Carnegie Mellon University\\Pittsburgh PA 15213\\U.S.A.}

\title{Edge disjoint Hamilton cycles in random digraphs of constant minimum degree}
\maketitle
\begin{abstract}
\parindent 0 in
We study the existence of  directed Hamilton cycles in random digraphs with $m$ edges where we condition on minimum in- and out-degree $\d \ge k+1$, where $k \ge 1$. Denote such a random graph by $D_{n,m}^{(\delta\geq k+1)}$.  Let  $m=cn$ and $c\ge c_k$, where $c_k$ is a sufficiently large constant.  We prove that   w.h.p. $D_{n,m}^{(\delta\geq k+1)}$ contains $k$ edge disjoint Hamilton cycles.
\end{abstract}

\section{Introduction}
Let $D_{n,m}$ denote the random digraph with vertex set $[n]$ and $m$ random edges. McDiarmid \cite{McD} proved that if $m=n(\log n+\log\log n+\om)$ where $\om\to\infty$ then $D_{n,m}$ is Hamiltonian w.h.p. Subsequently, Frieze \cite{F1} sharpened this and proved that if $m=n(\log n+c_n)$
\beq{H1}{
\lim_{n\to\infty}\Pr(D_{n,m}\text{ is Hamiltonian})=\begin{cases}0&c_n\to-\infty.\\e^{-2e^{-c}}&c_n\to c.\\1&c_n\to\infty.\end{cases}
}
The right hand side of \eqref{H1} is the limiting probability that $\d=\min\set{\d^-,\d^+}\geq 1$, where $\d^-$ (resp. $\d^+$) denotes the minimum in-degree (resp. out-degree) of $D_{n,m}$. Cooper and Frieze \cite{CFD1} proved the following: let $m=\tfrac n2(\log n+2\log\log n+c_n)$ then
\beq{H2}{
\lim_{n\to\infty}\Pr(\dnm{1}\text{ is Hamiltonian})=\begin{cases}0&c_n\to-\infty.\\e^{-e^{-c}/8}&c_n\to c.\\1&c_n\to\infty.\end{cases}
}
The right hand side of \eqref{H2} is the limiting probability that $\dnm{1}$ contains two vertices of in-degree one  share a common in-neighbour, or that two vertices of out-degree one share a common out-neighbour.

In this paper we consider random digraphs with $\d\geq k+1$ where $k\geq 1$, but with only  a linear number of edges $m=cn$, where $c$ is a constant  sufficiently larger than $k+1$. The condition that $\d\geq k+1$ is necessary for directed Hamilton cycles; because otherwise, with constant probability there are obstructions such as vertices of out-degree one with at least two in-edges.  We  prove the following theorem.

\begin{theorem}\label{main}
For every $k\geq 1$ there exists a  constant $c_k>0$ such that for any constant $c \ge c_k$, if $m=cn$ then   w.h.p. $D_{n,m}^{(\delta\geq k+1)}$ contains $k$ edge disjoint Hamilton cycles.
\end{theorem}

Previous work in this area has focused on undirected random graphs $G_{n,m}$. Let $\gnm{k}$ denote the space of simple graphs with $n$ vertices, $m$ edges and minimum degree at least $k$, each graph $G$ being equiprobable. Beginning with Bollob\'as, Cooper, Fenner and Frieze \cite{BCFF} there has been several attempts to find out how many random edges we need  for a Hamilton cycle if we condition on minimum degree at least three. Anastos \cite{A}  proved that $\gnm{4}$ is Hamiltonian w.h.p. if $m>2n$ and Anastos and Frieze \cite{AF2} have shown that $\gnm{3}$ is Hamiltonian w.h.p. if $m>2.663n$. 

Because a random cubic graph is Hamiltonian w.h.p. (see Robinson and Wormald \cite{RW1}), it is natural to conjecture that $\gnm{3}$ is Hamiltonian w.h.p. if $m\geq 3n/2$. More boldly, we suggest the following to be true for $k \ge 3$:\;\; 
Let $G\in \gnm{k}$ have property ${\cal A}_k$ if $G$ contains $\rdown{  (k-1)/2 }$ edge disjoint Hamilton cycles, and, if $k$ is even, a further edge disjoint matching of size $\rdown{ n/2 }$. For $k \ge 3$, ${\cal A}_k$ occurs in $\gnm{k}$ w.h.p. when $m\ge kn/2$.

\subsection*{Proof outline}
We will base our proof on the 3-phase method which has been  used in various forms in, for example,  \cite{CF1}, \cite{CFD1}, \cite{DF}, \cite{FKR}, \cite{FS}. In Phase 1 we construct $k$ sets $\cC_1,\cC_2,\ldots,\cC_k$, where each set $\cC_i$ consists of of $O(\log n)$ vertex disjoint directed cycles that cover all vertices. Furthermore the sets $\cC_i, \cC_j$ will be edge-disjoint for $i \ne j$. In Phase 2, we transform each of the $\cC_i$ so that they remain edge disjoint and each cycle in each $\cC_i$ has length at least $n/\log n$. In Phase 3, we {\em patch} transform these cycle covers into Hamilton cycles. Prior to this we describe a useable model for $\dnm{k+1}$, and  a Phase 0 in which we partition the edge set $E$ into $E_1,E_2,E_3,E_4$ for use in the corresponding phases.

\section{ Preparatory material.}

\subsection{A useable model for $\dnm{k+1}$}
For a sequence $\bx=(x_1,x_2,\ldots,x_{2m})\in [n]^{2m}$ we let $D_\bx$ be the multi-digraph with vertex set $[n]$ and edge set $E_\bx=\set{(x_{2j-1},x_{2j}):j=1,2,\ldots,m}$. Let  $d^+_{\bx}(i)=|\set{j:x_{2j-1}=i}|$ be the out-degree of $i\in[n]$ in $D_\bx$ and similarly let  $d^-_{\bx}(i)=|\set{j:x_{2j}=i}|$ be the in-degree of $i\in[n]$ in $D_\bx$.
 Then  $\Omega_k(m)=\set{\bx\in [n]^{2m}:d^+_{\bx}(i),d^-_{\bx}(i)\geq k+1 \text{ for all }i\in[n]}$. The degree $d_\bx(i)=d^+_{\bx}(i)+d^-_{\bx}(i)$. For brevity we will sometimes speak of $\bx$ as if it is a digraph $D_\bx$.
For example, the out-degree of $v \in \bx$ should be taken to mean the out-degree of $v$ in $D_\bx$ etc.

Let $d^\pm$ denote a given degree sequence, thus  $d^{\pm}=((d^-(j),d^+(j)),\; j \in [n])$,
and let $\Om_k(m,d^\pm)$ denote those $\bx \in \Om_k(m)$ whose digraph $D_\bx$ has  degree sequence $d^\pm$.
We  can generate a random member $\bx$ of $\Omega_k(m, d^\pm)$ using a bipartite version of Bollobas's configuration model \cite{Bol1}. We take a random permutation of the multi-set $\set{d^+(j)\times j:j\in [n]}$ and placing the values in $x_1,x_3,\ldots,x_{2m-1}$ and then take a random permutation of the multi-set $\set{d^-(j)\times j:j\in [n]}$ and placing the values in $x_2,x_4,\ldots,x_{2m}$.  (Here the notation $d\times a$  means that the  corresponding  multi-set contains $d$ copies of $a$.)

For $\bx\in\Omega_k(m)$,  let $L_\bx=\set{j\in[m]:x_{2j-1}=x_{2j}}$ be the set of loops in $D_\bx$, and  let $M_\bx=\{j\in[m]:\exists j'\neq j\text{ s.t. }(x_{2j'-1},x_{2j'})=$ $(x_{2j-1},x_{2j})\}$ define the set of multiple edges in $D_\bx$.
Let $\Omega_{k+1}^*=\set{\bx\in \Omega_{k+1}:L_\bx=M_\bx=\emptyset}$ denote those sequences whose associated digraph is simple. We have that
\beq{sizeD}{
D\in \cdnm{k+1}\text{ implies that }|\set{\bx\in \Omega^*_{k+1}:D_\bx=D}|=m!,
}
as given $\bx=(x_1,x_2,\ldots,x_{2m})$ each permutation of the sequence $s_\bx=((x_{2j-1},x_{2j}):j=1,2,\ldots,m)$ is distinct. Consequently, choosing $\bx$ uniformly from $\Omega_k^*$ and taking $D_\bx$ is equivalent to sampling uniformly from $\cdnm{k+1}$. We show below, see Lemma \ref{LxMx}, that if $m=O(n)$ then $L_\bx=M_\bx=\es$ with probability $\Omega(1)$. Consequently, we only need to show that $D_\bx$ has $k$ edge disjoint Hamilton cycles w.h.p. for the relevant values of $c_k$; as properties which hold w.h.p. in $\Omega_k$ will hold w.h.p. in $\Omega_k^*$.

\paragraph{Generating a random member of $\Omega_k(m)$.}
We  generate the in- and out-degree of  any given $i \in [n]$ as  independent {\em truncated Poisson} random variables, and then condition on total in- and out-degree $m$  using Lemma \ref{Om1} below.

Let $Z=Z(z)$ be a {\em truncated Poisson} random variable with parameter $z$, taking values in $\{ k+1,...\}$. Thus
\beq{Pz=j}{
\Pr(Z=j)=\frac{z^j}{j!f_{k+1}(z)},\hspace{1in}j=k+1,\ldots\ ,
}
where
\[
f_\ell(z)=\sum_{j=\ell}^\infty \frac{z^j}{j!}.
\]
The mean and variance of  $Z(z)$ are
\begin{align*}
\E(Z)&=\r(z)=\frac{zf_{k}(z)}{f_{k+1}(z)},\\
 \Var(Z)&=\s^2(z)=\frac{z^2f_{k-1}(z)}{f_{k+1}(z)}+\frac{zf_{k}(z)}{f_{k+1}(z)}-\frac{z^2f_k(z)^2}{f_{k+1}(z)^2}=O(1).
\end{align*}
We will choose a value for $z$ so that $\r(z)=c$, where $m=cn$. For the following details see e.g.   Lemma A.1 of \cite{BCFF}. That there is a unique value of $z$ such that $\r(z)=c$ for any $c>0$, follows from the fact that the function $\r(z)/z$ is monotone decreasing in $z$ and tends to one from above as $z \rai$, and moreover $z$ satisfies
\beq{c-k}{
c-(k+1)< z <c.
}
We note that if $z$ is sufficiently large and $k$ is constant then
\beq{ez}{
f_{k+1}(z)\geq \tfrac12 e^z,
}
so that probabilities given by \eqref{Pz=j} are at most twice those of an unconditioned $Po(z)$.

The proof of the following lemma is derived from \cite{AFP}
\begin{lemma}\label{Om1}
Let $Z_i',\;Z_i'',i\in [n]$ be independent copies of a truncated Poisson random variable $Z(z)$.
Let $\bx$ be chosen randomly from $\Omega_k$. Then $\{d_\bx^+(i)\}_{j\in [n]}$ is distributed as $\{Z_i'\}_{i\in [n]}$ conditional on $\sum_{j\in [n]}Z_i'=m$ and similarly for $\{d_\bx^-(i)\}_{i\in [n]}$ with respect to $Z_i''$.
\end{lemma}
\begin{proof}
Let
$$S = \Bigl\{ \bd \in [n]^n \,\Big|\sum_{1\leq j \leq n} d_j = m \text{ and } d_j \geq k+1 \text{ for all } j \Bigr\}.$$
Fix $\boldsymbol{\xi} \in S$.   Then, by the definition of $\bx$ and $d_\bx$,
$$\Pr(d_\bx = \boldsymbol{\xi}) =\left( \frac{m!}{\xi_1! \xi_2! \ldots \xi_n! }\right)\bigg/\left( \sum_{x\in S} \frac{m!}{x_1! x_2! \ldots x_n!} \right).$$
On the other hand, if $\bZ=(Z_1,Z_2,\ldots,Z_n)$ where the $Z_j$ are independent copies of the truncated Poisson random variable $Z$ given in \eqref{Pz=j}, then
\begin{align*}
\Pr \left(\bZ = \boldsymbol{\xi} \; \bigg|\; \sum_{1\leq j \leq n} Z_j =  m\right)
&=\left( \prod_{1\leq j\leq n}{z^{\xi_j}\over  \xi_j! f_{k+1}(z)}\right)
\bigg/\left( \sum_{x\in S} \prod_{1\leq j\leq n}{z^{x_j}\over   x_j! f_{k+1}(z)}\right)\\
&=\left(   { f_{k+1}(z)^{-n}z^{m}\over \xi_1! \xi_2! \ldots \xi_n!} \right)
\bigg/\left( \sum_{x\in S} {f_{k+1}(z)^{-n} z^{m} \over x_1! x_2! \ldots
x_n!} \right)\\
&=\Pr(d_\bx=\boldsymbol{\xi}).
\end{align*}
\end{proof}
It follows from the local central limit theorem, see for example Durrett \cite{Durret}, that
\beq{CLT}{
\Pr\brac{\sum_{j\in [n]}Z_i'=m}=\Pr\brac{\sum_{j\in [n]}Z_i’'=m}=\frac{1}{\sqrt{2\p \s^2n}}(1+O(n^{-1}\s^{-2})).
}
And that for $k=O(n^{1/2}\s)$,
\beq{CLT1}{
\Pr\brac{\sum_{j\in [n]}Z_i'=m-k}=\Pr\brac{\sum_{j\in [n]}Z_i’'=m-k}=\frac{1}{\sqrt{2\p \s^2n}}(1+O((k^2+1)n^{-1}\s^{-2}))
}
We use the following standard inequalities for the tails of the binomial distribution:
\begin{eqnarray}
\Pr(\card{B(n,p)-np}\geq \epsilon np) & \leq & 2e^{-\epsilon^ 2np/3},\hspace{.25in}0\leq \epsilon \leq 1, \label{Ch} \\
\Pr(B(n,p)\geq anp) & \leq & (e/a)^{anp}. \label{LD}
\end{eqnarray}

\subsection{Degree sequence properties of $\Omega_k(m)$}
For $\bx \in \Om_k(m)$ let  $\n_\bx(r,s)=|\set{i \in [n] :d_\bx^-(i)=r,\;d_\bx^+(i)=s}|$ count the vertices of $D_{\bx}$ with in-degree $r$ and out-degree $s$. Let $\cD_1=\set{\n_\bx(r,s)=0\text{ for }r+s>\log n}$, and let
\begin{equation} 
\label{degconc}
\cD_2 (r,s)=\left\{\left|\n_\bx(r,s)-\frac{nz^{r+s}}{ r!s!f_{k+1}(z)^2}\right| \leq K_1\brac{1+\bfrac{nz^{r+s}}{r!s!f_{k+1}(z)^2}^{1/2}}\log n \right\},
\end{equation}
where $z$ is such that $\r(z)=m/n=c$. Let $\cD_2 = { \cup_{k+1 \le r,s}\cD_2(r,s) }$, and let $\cD=\cD_1\cap\cD_2$.

Following Knuth, Motwani, and Pittel \cite{KMP}, we say that
an event $\cE=\cE(n)$ occurs quite surely (q.s., in short) if $\Pr(\cE)=1-
O(n^{-a})$ for any constant $a>0$.
Lemma \ref{Om1}, with (\ref{CLT}) and (\ref{CLT1}) plus our tail estimates \eqref{Ch}, \eqref{LD} for the binomial
distribution show that there is an absolute constant $K>0$ such that the event $\cD=\cD_1\cap\cD_2$ occurs q.s.
\begin{lemma}\label{LxMx}
$\Pr(L_\bx=M_\bx=\es)=\Omega(1)$.
\end{lemma}
\begin{proof}
We use the method of moments to show that $|L_\bx|,|M_\bx|$ are asymptotically independently Poisson with bounded mean. We first rule out some technicalities. Let $\cF_1$ be the event that there are two loops on the same vertex and let $\cF_2$ be the event that there is a triple edge. Let $\cB=(\neg\cD)\cup\cF_1\cup\cF_2$. Then as $\cD$ occurs q.s., $\Pr(\cB)=o(1)$ follows from
\begin{align*}
\Pr( \cF_1 \cap \cD)&\leq \sum_{i=1}^{n}\binom{m}{2}\frac{d_\bx(i)^4}{m^4}\leq \frac{n \log^4 n}{m^2}=o(1).\\
\Pr( \cF_2 \cap \cD)&\leq  \sum_{1\leq i<j\leq n}\binom{m}{3}\frac{d_\bx(i)^3d_\bx(j)^3}{m^6}\leq \frac{ n^2\log^6n}{m^3}=o(1).
\end{align*}
{\bf Explanation:} we justify the bound on $\Pr( \cF_1 \cap \cD)$. Fix a degree sequence in $\cD$. If there is a double loop then there are two edge indices $p,q$ and a vertex index $i$ such that $x_{2p-1}=x_{2p}=x_{2q-1}=x_{2q}=i$. The $\binom{m}{2}$ counts the choices of $p,q$ and for a fixed $i$. the probability that each of the 4 components equal $i$ is bounded by $(d_\bx(i)/m)^4$. The bound on $\Pr( \cF_2 \cap \cD)$ is proved similarly.

If we condition on the degree sequence $d^\pm$, then as observed above,  then a random  digraph $D_\bx$, $\bx \in \Om_k(m,d^\pm)$ can be generated in a bipartite version of the configuration model.

Now fix a positive integer $\ell$ and only consider single loops and double edges. Then
\begin{align*}
\E\brac{\binom{|L_\bx|}{\ell}}&=\frac{1}{\ell!}\sum_{\substack{S\subseteq [n]\\|S|=\ell}}\frac{1}{m_{(\ell)}}\prod_{i\in S}{d_\bx^+(i)d_\bx^-(i)}\\
&=\frac{1}{\ell!}\brac{\sum_{i=1}^n\frac{d_\bx^+(i)d_\bx^-(i)}{m}}^\ell+O\bfrac{\log^4n}{m}\\
&\sim \frac{1}{\ell!}\brac{\sum_{r,s=k+1}^{\infty}\frac{1}{c}\frac{r s z^{r+s}}{ r!s!f_{k+1}(z)^2}}^\ell=\frac{\r^{2\ell}}{c^\ell\ell!},
\end{align*}
Thus,
\beq{Lpos}{
\Pr(L_\bx=\es)\sim e^{-\r^2/c}.
}
We now show that $|M_\bx|$ is also asymptotically Poisson and asymptotically independent of $L$.
So, let $\ell=O(1)$ and let $\cM_\ell$ denote the set of collections of configuration points making up $\ell$
double edges and let $F$ denote the whole random pairing. Then
\begin{align*}
\E\brac{\binom{|M_\bx|}{\ell}\bigg|\,|L_\bx|=0}&=\sum_{\cM_\ell}\Pr(\cM_\ell\subseteq F\mid |L_\bx|=0)\\
&=\sum_{\cM_\ell}\frac{\Pr(|L_\bx|=0\mid \cM_\ell\subseteq F)\Pr(\cM_\ell\subseteq F)}{\Pr(|L_\bx|=0)}.
\end{align*}
Let
\[
\cC=\set{(S,\f(S)):S,T \seq[n],|S|=|T|=\ell, T=\f(S), \f \text{ bijection}}
\]
Now because $\ell=O(1)$, we see that the calculations that give us \eqref{Lpos} will give us\\
 $\Pr(|L_\bx|=0\mid \cM_\ell\subseteq F)\sim {\Pr(|L_\bx|=0)}$. So,
\begin{align*}
\E\brac{\binom{|M_\bx|}{\ell}\bigg|\,|L_\bx|=0}&\sim \frac1{\ell!}\sum_{\cM_\ell} \Pr(\cM_\ell\subseteq F)\\
&=\frac1{\ell!} \sum_{(S,\f(S))\in \cC} \;\frac1{m_{(2\ell)}}\;\prod_{i\in S} {2\binom{d_\bx^+(i)}{2}\binom{d_\bx^-(\f(i))}{2}}\\
&=\frac{1}{\ell!}\sum_{i,j=1}^n\bfrac{(d_\bx^+(i))_{(2)}(d_\bx^-(j))_{(2)}}{2m^2}^{\ell}+ O\bfrac{\log^8n}{m}\\
&\sim \frac{1}{\ell!}\brac{\sum_{r,s=k+1}^{\infty}\frac{r_{(2)} s_{(2)}z^{r+s}}{2c^2 r!s!f_{k+1}(z)^2}}^{\ell}
=\frac{1}{\ell!}\bfrac{z^2f_{k-1}(z)}{cf_{k+1}(z)}^{2\ell}.
\end{align*}
Thus,
\beq{Lpos1}{
\Pr(M_\bx=\es\mid L_\bx=\es)\sim e^{-z^2f_{k-1}(z)/cf_{k+1}(z)}.
}
and the lemma follows from \eqref{Lpos} and \eqref{Lpos1}.
\end{proof}
The next lemma bounds the degree of a set of vertices in $D_\bx$.
\begin{lemma}\label{degs}
Let $\eta=ez$, and define the event $\cF$ by
\beq{N(S)}{
\cF=\set{d_\bx(S)\leq \eta|S|\log(n/|S|) \text{ for all $S$ with }k\leq |S|\leq n/2.}
}
Then $\cF$ holds w.h.p.
\end{lemma}
\begin{proof}
\begin{align*}
\Pr(\neg \cF)&\leq 2\sum_{s=k}^{n/2}\binom{n}{s}\sum_{D\geq \eta s\log(n/s)}\;\sum_{d_1+\cdots+d_s=D}\;\prod_{i=1}^s\frac{z^{d_i}}{d_i!f_k(z)}\\
&=2\sum_{s=k}^{n/2}\binom{n}{s}\frac{1}{f_k(z)^s}\sum_{D\geq \eta s\log(n/s)}\frac{z^D}{D!}\sum_{d_1+\cdots+d_s=D}\frac{D!}{d_1!d_2!\cdots d_s!}\\
&=2\sum_{s=k}^{n/2}\binom{n}{s}\frac{1}{f_k(z)^s}\sum_{D\geq \eta s\log(n/s)}\frac{z^Ds^D}{D!}\\
&\leq2\sum_{s=k}^{n/2}\binom{n}{s}\frac{1}{f_k(z)^s}\sum_{D\geq \eta s\log(n/s)}\bfrac{zes}{D}^D\\
&\leq3\sum_{s=k}^{n/2}\brac{\frac{ne}{sf_k(z)}\cdot\bfrac{ze}{\eta \log(n/s)}^{\eta\log(n/s)}}^s\\
&\leq 3\sum_{s=k}^{n/2}\brac{\frac{2ne^{2+k-c}}{s}\cdot\bfrac{1}{ \log(n/s)}^{e(c-(k+1))\log(n/s)}}^s=o(1).
\end{align*}
In the last line we used \eqref{ez} and \eqref{c-k}.
\end{proof}

\subsection{Phase 0. An edge partition}\label{phase0}

Our input will be a random $\bx\in \Omega_k(m)$. Phase 0  partitions the $m$ edges of $E_\bx$ into four subsets $E_1,E_2,E_3,E_4$ each with different purpose.

For $i=1,2,\ldots,3k$ let $\wh E_{1,i}$ be a random subset of $[m]\sm \bigcup_{j=1}^{i-1}\wh E_{1,j}$ where each edge is chosen independently with probability $p_i=1/(4k-i+1)$. Note that $\E(|\wh E_1|)=m/4k$ and that for $i=2,\ldots,k$,
\begin{equation}\label{E-size}
\E(|\wh E_{1,i}|)=mp_i\prod_{j=1}^{i-1}(1-p_j)=\frac{m}{4k-i+1}\prod_{j=1}^{i-1}\frac{4k-j}{4k-j+1}=\frac{m}{4k}.
\end{equation}
After this we rename $\wh E_{1,k+i}$ as $\wh E_{2,i}$ for $i=1,2,\ldots,k$ and we rename $\wh E_{1,2k+i}$ as $\wh E_{3,i}$ for $i=1,2,\ldots,k$. Finally we define $E_4=E_\bx\setminus (\wh E_1\cup \wh E_2\cup \wh E_3)$.

Recall that $c=m/n$ is the average expected degree, and let \SM\ denote the set of vertices with in-degree or out-degree at most $c/8k$ in $D_\bx$ or in $\wh E_{t,i}$ for some $1\leq t\leq 3,1\leq i\leq k$.

Let $E_\SM$ denote the set of edges incident with \SM\ in $D_\bx$ and let $E_{j,i}=\wh E_{j,i}\cup E_\SM$ for $j=1,2,3$.
\begin{lemma}\label{smalldegree}
$|\SM|\leq ne^{-c/100k}$ w.h.p. for all $i$.
\end{lemma}
\begin{proof}
Before conditioning on $|\bx|=2m$ we have that each vertex has the distribution given by \eqref{Pz=j}. Let $W_i$ be distributed as $Z$ given in \eqref{Pz=j}, and let $Y_i=Bin(W_i,p_i)$. Now $\Pr(Z\leq c/8k)$
 is easily shown to be less than $e^{-c/99k}$.  The Chernoff bound then shows that $\Pr(|\SM|\geq 2ne^{-c/99k})\leq e^{-\Omega(n)}$.  To account for conditioning on $m=cn$, we inflate this upper bound by $O(n)$ using \eqref{CLT}, \eqref{CLT1}. This proves the lemma.
\end{proof}
Let $W$ be distributed as $Z$ given in \eqref{Pz=j}, and let $Y=Bin(W,p)$ for $0<p<1$. Fix $s$  and let $\p(s,D)=\Pr(Y_1+Y_2+\cdots+Y_s=D)$ where $Y_1,Y_2,\ldots,Y_s$ are independent copies of $Y$.  If $d_i$ is the value of $Y_i$ then

\begin{align}
\p(s,D)&= \sum_{d_1+\cdots+d_s=D}\prod_{i=1}^s\sum_{w_i\geq d_i \atop w_i \ge k+1}\frac{z^{w_i}}{w_i!f_{k+1}(z)}\,\binom{w_i}{d_i}p^{d_i}(1-p)^{w_i-d_i}\label{PoutS}\\
&\leq\sum_{d_1+\cdots+d_s=D}\prod_{i=1}^s\frac{(pz)^{d_i}}{f_{k+1}(z)d_i!}\sum_{w_i\geq d_i}\frac{(z(1-p))^{w_i-d_i}}{(w_i-d_i)!}\nn\\
&=\sum_{d_1+\cdots+d_s=D}\prod_{i=1}^s\frac{(pz)^{d_i}e^{(1-p)z}}{f_{k+1}(z)d_i!}\nn\\
&=\frac{(pz)^De^{(1-p)zs}}{f_{k+1}(z)^sD!}\sum_{d_1+\cdots+d_s=D}\frac{D!}{d_1!d_2!\cdots d_s!}\nn\\
&=\frac{e^{(1-p)zs}}{f_{k+1}(z)^s}\frac{(pzs)^D}{D!}.\label{Prob}
\end{align}

As  $D_\bx$ has exactly $m$ edges, we also need the following result.
\begin{lemma}\label{degS}
Let $Y_\ell, \ell \in [n]$ be distributed as in the definition of $\p(s,D)$.
Let $Z_\ell,\ell\in[n]$ be truncated Poisson variables, as given in \eqref{Pz=j}. Fix a set $X\subseteq \{1,3,\ldots,2m-1\}$ where $i\in X$ independently with probability $p$.
Then  $\wh \p(s,D)=\p(s,D)/ P_{s,D,m}$ upper bounds the probability that a set $S \seq \{1,3,\ldots,2m-1\}$ of size $s$ has $D$ total occurrences in the set $X$,
where
\beq{bpsdm}{
P_{s,D,m}=\begin{cases}\Omega(n^{-1/2})&\text{for all }s.\\
1+o(1)&s\leq \log^2n.\end{cases}
}
\end{lemma}
\begin{proof}
The display \eqref{PoutS}-\eqref{Prob} counts the out-degree $Y_j$ of  vertices $j \in S$, generated by  truncated Poisson variables $Z_j$, where each edge is subsequently retained with probability $p$. This is the probability that the odd coordinate lies in $X$. Thus $\p(s,D)$ is the probability the final total out-degree of the edges of $S$ is $D$. We must now account for conditioning on $\sum_{i=1}^nZ_i=m$.

As $\p(s,D)=\sum_{\D \ge D} \Pr(\sum_{i \in S}Z_i=\D  \text{ and } \sum_{i \in S} Y_i=D)$,
\[
\wh \p(s,D)=  \sum_{\D \ge D} \Pr\brac{\sum_{i \in S}  Z_i=\D  \text{ and } \sum_{i \in S} Y_i=D }\;
 \frac{\Pr(\sum_{j \not \in S} Z_j=m-\D)}{\Pr(Z_1+\cdots+Z_n=m)} \le \frac{\p(s,D)}{ P_{s,D,m}},
 \]
say.
By \eqref{CLT}, $\Pr(Z_1+\cdots+Z_n=m) =\Om(1/\sqrt n)$ for any $s$, lower bounding $P_{s,D,m}$. For small $s$, as $\D \le s \log n$ w.h.p.,
then for $s =o(\sqrt n /\log n)$ using \eqref{CLT} and \eqref{CLT1} gives $P(s,D,m)=1+o(1)$.
\end{proof}
\section{Proof of Theorem \ref{main}}
In the first phase we show that w.h.p. there is a collection $\c{C}_1$ of $k$ edge-disjoint cycle covers, $\{\cC_{1,1},... \cC_{1,k}\}$. Each $\cC_{1,i}$ comprises $O(\log n)$ vertex disjoint directed cycles (a {\em cycle cover} being a set of vertex disjoint cycles that include every vertex). Then in the second phase we transform $\c{C}_1$ into a collection of (similarly indexed)  cycle covers $\c{C}_2$ in which each cycle $C\in \c{C}_2$ has size at least $n/\log n$. We then transform $\cC_2$ into a collection of $k$ edge disjoint Hamilton cycles.
\subsection{Phase 1. Edge disjoint perfect matchings}
Our aim here is to construct $k$, edge disjoint perfect matchings $M_i,i\in[k]$.
For each $M_i$ we do this in steps using the edge sets $E_{1,i}, E_{2,i}$.
In summary, we will prove the following lemma:

\begin{lemma}\label{added}
W.h.p., $E_1\cup E_2\cup E_\SM$ contains $k$ edge disjoint perfect matchings $M_1,M_2,\ldots,M_k$. Individually, each matching is uniformly random, but obviously not collectively.
\end{lemma}
There is a natural bijection between digraphs and bipartite graphs. Given a digraph $D$ with vertex set $[n]$, we can define a bipartite graph $G(D)$ with vertex partition $=\set{a_i:i\in[n]}$, $B=\set{b_j:j\in [n]}$ and an edge $\set{a_i,b_j}$ for every edge $(i,j)$ of $D$. If $D$ satisfies $\min\set{\d^-,\d^+}\geq k+1$ then $G(D)$ has minimum degree at least $k+1$ and vice-versa. Let $G_\bz=G(D_\bz)$ for $\bz\in \Omega_k$.

So now, let $G_{i},i=1,2,\ldots,k$ be bipartite graphs with vertex sets $A=[n]$ and $B=[n]$. The edges of $G_{i}$ consist of the $A:B$ edges of 
$(E_{1,i}\cup E_{\SM})\sm \bigcup_{j=1}^{i-1}M_j$.
Thus degrees in $ G_i$  must be at least $k+2-i$, as they have dropped by at most $i-1$ due to the deletion of $M_1,\ldots,M_{i-1}$.

A set of vertices $S$ is monocolored if $S\subseteq A$ or $S\subseteq B$. For  monocolored $S$,  define $N_i(S)$ to be the set of neighbors of $S$ in $G_{i}$ respectively. We let $D$ denote the total degree of $S$ in $G_i$ and we let $\wh D\geq D$ denote the total degree of $N_i(S)$. As every vertex has at least $k+2-i$ remaining neighbours, we can assume that $|S|=s\geq k+2-i$; and similarly   we can assume $D\geq (k+2-i)s$

The existence of $M_i$ will follow from a modification of Hall's theorem.
\begin{lemma}\label{pm}
Let $\z_1$ be a fixed positive constant, sufficiently small to justify any inequality below. Let $\cH_i$ be the event that there  exists a monocoloured set $S$ in $G_i$ with $|S|\leq \zeta_1n$ such that $|N_i(S)|<|S|$. Then
\beq{M1}{
\Pr\brac{\bigcup_{i=1}^k\cH_i}=o(1).
}
\end{lemma}
\begin{proof}
Referring to Lemma  \ref{degS} above, $\wh \p(s,D)=\p(s,D) / P_m$. Then, with $p=p_i$,
\beq{cH}{
\Pr(\cH_i)\leq \frac 2{P_m^2} \sum_{s=k+2-i}^{\z_1n}\binom{n}{s}\binom{n}{s-1}
\sum_{\wh D\geq D\geq (k+2-i)s}\p(s,D)\p(s-1,\wh D)\bfrac{\wh D}{m}^D
}
We first obtain an expression for  \eqref{cH} ignoring the $1/P_m$ factor, and subsequently apply Lemma \ref{degS} to give the required upper bound.

The contribution of $D$ to \eqref{cH} is $f(D,\wh D)=(pzs)^D\wh D^D/(D!m^D)$.  Using Lemma \ref{degs},
\[
\frac{f(D+1,\wh D)}{f(D,\wh D)}=\frac{pzs \wh D}{(D+1)m}\leq \frac{\eta pcs^2\log(n/s)}{(k+2-i)sm}\leq \frac12,
\]
for small enough $\zeta_1$. In which case we can fix $D=d_0=(k+2-i)s$  and
\begin{equation}\label{Rats}
\sum_{d \ge d_0} f(d,\wh D) \le O(1) \;\frac{(pzs)^{d_0}}{ d_0!} \;\bfrac{\wh D}{m}^{d_0}.
\end{equation}
We define $g(d)=(pzs)^d d^{d_0}/d!$ and then observe that for $d\geq d_0$,
 \begin{align*}
 \frac{g(d+1)}{g(d)}=&\frac{pzs}{d+1}\bfrac{d+1}{d}^{d_0}\text{ and so }\frac{g(d)}{g(d_0)}\leq \frac{d_0!(cps)^{d-d_0}\bfrac{d}{d_0}^{d_0}}{d!}.
 \end{align*}
Now $\wh D \ge D = d_0$. So,
 \begin{align*}
 \sum_{\wh  D\geq d_0}\p(s-1,\wh D)\bfrac{\wh D}{m}^{d_0}&= \frac{e^{(1-p)zs}}{f_{k+1}(z)^{s-1}m^{(k+2-i)s}}\sum_{d\geq d_0}g(d)\\
 &\leq \frac{e^{(1-p)zs}g(d_0)}{f_{k+1}(z)^{s-1}m^{(k+2-i)s}} \sum_{d\geq  d_0}\frac{d_0!(cps)^{d-d_0}\bfrac{d}{d_0}^{d_0}}{d!}\\
 &\leq \frac{e^{(1-p)zs}g(d_0)}{f_{k+1}(z)^{s-1}m^{(k+2-i)s}} \sum_{d\geq  d_0}\frac{d_0!(cps)^{d-d_0}e^{d/e}}{d!}\\
 &\leq \frac{e^{(1-p)zs}e^{cps+e^{1/e}+d_0/e}}{f_{k+1}(z)^{s-1}m^{(k+2-i)s}}\cdot\frac{g(d_0)d_0!}{(cps)^{d_0}}\\
 &\leq
 \frac{e^{(1-p)zs}e^{cps+e^{1/e}+d_0/e}d_0^{d_0}}{f_{k+1}(z)^{s-1}m^{(k+2-i)s}}.
 \end{align*}

From \eqref{Rats} and \eqref{cH}, noting the bound on $\p(s,D)$ from \eqref{Prob}, with $d_0=(k+2-i)s$ we get
\begin{align}
\Pr(\cH_i)&\leq \sum_{s=k+2-i}^{\zeta_1n}\frac{(en)^{2s-1}}{s^s(s-1)^{s-1}}\cdot \frac{e^{
(1-p)zs}(pcs)^{(k+2-i)s}}{f_{k+1}(z)^{s}((k+2-i)s)!}\cdot \frac{e^{(1-p)zs}e^{cps+e^{1/e}+d_0/e}d_0^{d_0}}{e^{(k+2-i)s}f_{k+1}(z)^{s-1}m^{(k+2-i)s}}\nn\\
&\leq \sum_{s=k+2-i}^{\zeta_1n}\frac{n^{2s-1}s^{(k+2-i)s}}{n^{(k+2-i)s}s^s(s-1)^{s-1}} \bfrac{e^{2+ 2(1-p)c+cp+(k+2-i)(1+1/e)}p^{k+2-i}}{f_{k+1}(z)^{2-1/s}}^s\label{brackets}\\
&\leq \sum_{s=k+2-i}^{\zeta_1n}\frac{s^{1+(k-i)s}e^{-pcs/2}}{n^{1+(k-i)s}}=O(n^{-1+o(1)})+
\sum_{s=\log^2n}^{\zeta_1n}\frac{s^{1+(k-i)s}e^{-pcs/2}}{n^{1+(k-i)s}}=O(n^{-1+o(1)}),\nn
\end{align}
for $c$ sufficiently larger than $k$.  This is the unconditioned probability, and we apply Lemma \ref{degS} to give $\Pr\brac{\bigcup_{i=1}^k\cH_i}=o(1)$,  completing the proof of the lemma.

In \eqref{brackets} note that $e^{2+(k+2-i)(1+1/e)} p^{k+2-i)}$ can be bounded by a constant independent of $c$. The remaining terms inside the big bracket $e^{2(1-p)z+cp}/f_{k+1}(z)^{2-1/s}$ can be bounded by $e^{-pc/2}$.
\end{proof}
To complete the proof of Lemma \ref{added} we use the following lemma  (Lemma 6.3 of \cite{FK}) using the edges of $E_2$.
We note that if $\cH=\bigcup_{i=1}^k\cH_i$ is as defined in Lemma \ref{pm} and $\cH$ does not hold before adding edges of $E_2$, then a fortiori $\cH$ does not hold after adding edges. 
\begin{lemma}\label{boost}
Let $G$ be a bipartite graph without a perfect matching. Let $M$ be a maximum cardinality matching and suppose that $v$ is not covered by $M$. Let $B(v)$ be the set of vertices $w$ for which there exists some maximum matching (not necessarily $M$) that does not cover $v,w$. Then $|N_G(B(v))|<|B(v)|$. \qed
\end{lemma}
Suppose at some stage we have found $\ell-1<k$ edge disjoint perfect matchings $M_1,\ldots,M_{\ell-1}$ in $G_j, j < \ell$ and  a maximum matching $M'_\ell$ in $G_\ell$ that is not perfect and which we next try to augment. As before, to avoid using edges of \SM\ twice we  update 
$(E_{2,\ell}\cup E_{\SM})\sm \bigcup_{j=1}^{\ell-1}M_j$ at the start of augmenting $M'_\ell$.

It follows from Lemmas \ref{pm} and \ref{boost} that in $G_\ell$ there exists a set $X_A=\{x_1,x_2,\ldots,x_r\}\subseteq A,r\geq \z_1n$ and sets $Y_j \seq B,\,j=1,2,\ldots,r$, such that (i) $|Y_j|=r,\;j=1,2,\ldots,r$ and (ii) if $y\in Y_j$ then there is a maximum matching that isolates $x_j$ and $y$. So, if we can add any such edge $(x_j,y)$ then we can increase the maximum matching size.  We call such edges {\em boosters}.

Each matching $\ell=1,...,k$ has a separate edge set $E_{2,\ell}$ available to augment it. To augment $M'_\ell$  we examine the  edges of $E_{2,\ell}\sm E_\SM=\set{f_1,f_2,\ldots,f_\m}$, $\m\gtrsim m/4k$ in sequential order. An edge $f_j$ is a booster with probability at least $\eta_0\geq (r^2-n|\SM|)/n^2\gtrsim\zeta_1^2$. So there is a probability of at least $\eta_0$ that $f_j$ will increase the size of the current matching. To obtain a perfect matching, we require at most $n$ successes. 
However, for sufficiently large $m$, it holds that $\Pr(Bin(m/4k,\eta_0)<n)=o(1)$. To find $k$ perfect matchings we need 
to repeat this at most $k$ times,  and as $k$ is constant this means that w.h.p. we can add the required booster edges to construct $k$ perfect matchings.

\paragraph{The matchings are random}
 It can be shown by symmetry that these cycle covers implied by the perfect matching $M_i$ of $G_i$ can be taken to be uniform over  covers of their corresponding digraphs $D_i$. (Not of course independent.) Indeed, given $\wh D_i$ and a permutation $\p$ of  $V_i$ we define $\p D_i=(V_i,\set{(i,\p(j)):(i,j)\in E( D_i)}$. Let $\p$ be a uniform random permutation of $V_i$. Suppose that we check for a cycle cover in $D_i$ by looking for a cycle cover $\wh \cC_{1,i}$ in $\p D$ and then taking $\p^{-1}\wh \cC_{1,i}$ to be our cycle cover $\cC_{1,i}$ in $\wh D_i$. Note that $\p^{-1}\wh \cC_{1,i}$ is uniformly random, no matter what value $\cC_{1,i}$ takes. Finally note that $\p D_i$ has the same distribution as $D_i$ so that $\Pr(D_i\text{ has a cycle cover})=\Pr({\p D_i}\text{ has a cycle cover})$. A random cycle cover has at most $2\log n$ cycles w.h.p. and so we have that w.h.p. each of these $k$ cycle covers consist of at most $2\log n$ cycles.

\subsection{Phase 2. Minimum cycle length $n/\log n$}
In this phase we transform each $\wh\cC_i$  into a cycle cover $\cC_i^*$ in which each cycle has size at least $n/\log n$. We prove that Phase 2 succeeds w.h.p. as stated in the following lemma.
\begin{lemma}\label{Phase2}
With probability $1-o(1)$ at
the end of Phase 2 we have edge disjoint cycle covers $\wh\cC_i^{*},i=1,2,\ldots,k$ in which the minimum cycle length is at least $n_0=n/\log n$.
\end{lemma}
To prove Lemma \ref{Phase2} we use the edges of $E_3$ in an algorithmic construction  which we next describe. We denote the set of $E_{3,i}$-out-neighbors of $v$ by $out(v)$ and $E_{3,i}$-in-neighbors of $v$ by $in(v)$. Let $D_i$ denote the digraph induced by $E_{3,i}\cup M_i$.

Fix $i$ and partition the cycles $\cC_i$ associated with $M_i$ into {\em small} cycles $C$, where $|C| < n_0$ and {\em large} cycles $|C| \geq n_0$.  A cycle cover can be viewed as a  {\em Permutation Digraph} (PD). Starting with the PD $\Pi$ obtained from the current cycle cover,
we define a {\em Near Permutation Digraph} (NPD) to be a digraph obtained from  $\Pi$ by removing one edge. Thus an NPD $\U$ consists of a path $P(\U)$ plus a smaller permutation digraph $PD(\U)$ which covers $[n]\setminus V(P(\U))$.

In a random permutation the expected number of vertices on cycles of length at most $s$ is precisely $s$ (see e.g., \cite{K}). Let
\begin{equation}\label{CalM}
\cM_i =\{\cC_{i} \text{ contains at most } n\log\log n/\log n \text{ vertices on small cycles.} \}
\end{equation}
By the Markov inequality, $\cM_i$ holds w.h.p., and we condition on these $k$ events.

We now give an informal description of a process which removes a small cycle $C$ from a {\em current} PD $\U$. We start by choosing an arbitrary edge $(v_0,u_0)$ of $C$ and then delete it to obtain an NPD $\U_0$ with $P_0=P(\U_0)\in {\cal P}(u_0,v_0)$, where ${\cal P}(x,y)$ denotes the set of paths from $x$ to $y$ in $\G_{3,i}$.

The aim of the process is to produce an {\em out-tree} comprisng a large set $S$ of NPD's such that for each $\U\in S$, (i) $P(\U)$ has at least $n_0$ edges and (ii) the small cycles of $PD(\U)$ are a subset of the small cycles of $\Pi$. We will show that w.h.p. the endpoints of one of the $P(\U)$'s can be joined by an edge to create a permutation digraph with at least one less small cycle.

The Out-Phase consists of a sequence of {\em basic steps.} In a  basic step of the {\em Out-Phase} we have an NPD $\U$ where $P(\U)$ is a path from $u_0$ to another vertex $v$. We examine the edges leaving $v$, i.e., the edges going {\em out} from the endpoint $v$ of the path in $D_i$. Let $w$ be the terminal vertex of such an edge and assume that  $\U$ contains an edge $(x,w)$. Then $\U'=\U\cup\{(v,w)\}\setminus\{(x,w)\}$ is also an NPD with a path $P'$ from $u_0$ to $x$. We say that edge $(x,w)$ and  $\U'$ are {\em acceptable} if (i) $P'(\U')$ contains at least $n_0$ edges and (ii) any new cycle created (i.e. in $\U'$ and not $\U$) also has at least $n_0$ edges. We allow  $x  \in P(\U)$ in this construction. If $x  \in P(\U)$ this replaces the path $P$ by $P'$. If $x \not \in P$ then $P$ extends to $P'$, a longer path. In this context we call $v$ the {\em parent} vertex and $w$ the {\em child} vertex.

As each vertex $w$ has a unique predecessor $x$ in a cycle cover PD, the absence of the required predecessor edge $(x,w)$ in the NPD
$\U$ indicates the special case that $x=v$ and $w=u_0$.  We accept the edge if $P$ has at least $n_0$ edges. This would  successfully end the  iteration, although it is unlikely to occur.

We do not want to look at very many edges in this construction. To avoid this, we build a tree $T_0$ of NPD's in a natural breadth-first fashion where each non-leaf vertex in $\U$ gives rise to one or more NPD offspring $\U'$ as described above. The construction of $T_0$ ends when  the tree first has $\nu=\sqrt n\log n$ leaves. The construction of $T_0$ constitutes an Out-Phase of our procedure to eliminate small cycles.  Having constructed $T_0$ we follow with a further {\em In-Phase}, similar to the Out-Phase, using the other end of the set of paths. If an edge of the digraph forms  a cycle with one of the set of such paths obtained by acceptable moves, we have succeeded in removing a small cycle without creating any new ones. We prove that w.h.p. we  can close at least one of the paths $P(\U)$ to a cycle of length at least $n_0$.

We now increase the formality of our description. For $i=1,...,k$, we start Phase 2i with a PD $\Pi_0$, say, and a general iteration of Phase 2i starts with a PD $\Pi$ whose small cycles are a subset of those in $\Pi_0$. Iterations continue in Phase 2i until there are no more small cycles. At the start of an iteration we choose some small cycle $C$ of $\Pi$, and do the Out-Phase in which we construct a tree $T_0=T_0(\Pi,C)$ of NPD's as follows: the root of $T_0$ is $\U_0$ which is obtained by deleting an edge $(v_0,u_0)$ of $C$.

The set of nodes at depth $t$ is denoted by $S_t$. Let $\U \in S_t$ and $P=P(\U)\in {\cal P}(u_0,v)$. The {\em potential} offspring $\U'$ of $\U$, at depth $t+1$ are defined as follows.

Let $w$ be the terminal vertex of an edge directed from $v$.
\\
{\bf Case 1.} Path Extension. $w$ is a vertex of a cycle $C' \in PD(\U)$ with edge $(x,w) \in C'$.\\
Let $\U'=\U\cup\{(v,w)\}\setminus \{(x,w)\}$, thus extending the path $P$.
\\
{\bf Case 2}. Path Closure. $w$ is a vertex of $P(\U)$.\\
 Either $w = u_0$, or $(x,w)$ is an edge of $P$. In the former case $\U\cup\{(v,w)\}$ is a PD $\Pi'$ and in the latter case we let $\U'=\U\cup\{(v,w)\}\setminus \{(x,w)\}$, making a cycle and a shorter path in $\cP(u_0,x)$.

We  maintain a set $W$ of {\em used} vertices during  Phase 2. At the start of Phase 2,  all vertices  are {\em unused} and $W=\emptyset$. Whenever we examine an edge $(v,w)$, we add both $v$ and $w$ to $W$.
  We do not allow $|W|$ to exceed $n^{3/4}$.

We only admit to $S_{t+1}$ those $\U'$ which satisfy the following conditions, {\bf C}(i) and {\bf C}(ii).

{\bf C}(i) The new cycle formed (Case 2 only) must have
at least $n_0$ vertices, and the path formed (both cases) must either
be empty or have  at least $n_0$ vertices. When the path formed is empty we end the iteration and if necessary start the next with $\Pi'$.

{\bf C}(ii)  The vertices $x,w \not \in W$ at the start of iteration $t+1$.

 We stop an iteration $t$, in mid-phase if necessary, when $|S_t| \in[\nu,3\nu]$. Let us consider a generic step in the growth of $T_0$ where we are extending from $\U$ and $P(\U)\in{\cal P}(u_0,v)$.

The set $S_{t+1}$ is constructed in the following manner. We first examine $out(v), v\in S_t$ in the order that the vertices  $v$ were placed in $S_t$ to see if they produce acceptable edges. Note that in the general description, $v$ is the vertex $x'$ of edge $(x',w')$ broken at iteration $t$. Assuming $x' \not \in W$, then with the exception of the edge $(x',w')$, the out-edges of $v=x'$ are  unexamined.
Considering the edges $(v,w)$, $w \in out(v)$, we then add to $S_{t+1}$ those vertices $x\not\in W$ which arise from $(x,w)$ with $w\not\in W$.

A vertex $w\in out(v)$ is unacceptable if either (i) $w$ lies on $P(\U)$ and is too close to an endpoint so that the new path has length less than $n_0$; this has probability $O(1/\log n)$, or (ii) the corresponding vertex $x$ is in $W$; this has probability bounded above by $O(n^{-1/4})$, or (iii) $w$ lies on a small cycle; which, given the event $\cM_i$ defined in \eqref{CalM}, has probability  $O(\log \log n/\log n)$.

Every vertex will have out-degree at least $k+1-(i-1)\geq 2$ and so when we examine $v$ there is always at least one choice of $w$ that is acceptable w.h.p. Our agument relies on $S_t$ expanding and so we must in particular deal with the case where $w$ is unique. When we build our tree of NPD’s there is always a chance that it contains a path where the sequence $v_1,v_2,\ldots,v_k$ of parent vertices are all in \SM. In the worst-case, these vertices will be of degree 2 and there is no expansion. Such a path is said to be {\em induced}. It follows from Lemma \ref{smalldegree} that the probability there exists an induced path of length $\log\log n$ starting from some fixed NPD is at most $\th_1=\log^{-c/100k}n$ and this probability bound will suffice.

So, we can look at the growth of the out-tree in the following way. At any vertex, if necessary, we first produce an induced path of length at most $\log\log n$ that ends with a $v\notin\SM$. Treating the induced path as an edge, the tree will have a branching factor of at least $(1-o(1))\a$ where $\a=\rdup{c/8k}$. (The $o(1)$ term is of order $O(\log\log n)/\log n)$ due to $w$ being unacceptable. To simplify matters, we can prune the tree and make the branching factor equal to be at most $\a$. So, if grow the out-tree in a breadth first manner, level $t$ will have at least $((1-o(1))\a)^t$ nodes. This provided we stop growing the tree once we have $\n=n^{1/2}\log n$ nodes.

We need to check the probability that we actually achieve the claimed $\n$ nodes. Consider the first $\log\log n$ levels. First of at the root of the tree, there is an $O((\log\log n)^2/\log n)$ chance of failure. This is because at each step of an induced path there is an $O(\log\log n/\log n)$ chance that the child $w$ is unacceptable. So, we can now assume that $|S_1|\geq \a$.

Arguing this way, we see that $|S_{t+1}|\geq \a|S_t|/2$ for $1\leq t\leq \log\log n$ with probability $1-O((\log\log n)^2/\log n)$. So, w.h.p. $|S_{\log\log n}|\geq \log^{\a/2}n$. Now $|S_{t+1}|$ dominates $Bin(|S_t|, 1-O((\log\log n)^2/\log n)$ and so the Chernoff bounds imply that we reach $\n$ nodes with probability $1-O((\log\log n)^2/\log n)$. (Note that the expected growth of $|W|$ will be $O(\n\times \log\log n\times \log\log n/\log n)$. The factor $\log\log n$ comes from induced
paths. Applying the Markov inequality, we see that the increase is at most $\n\log^3n$ with probability at least $1-1/\log^2n)$.)

Therefore with probability $1-O((\log\log n)^2/\log n)$ we can end an Out-phase with out-tree $T_0$ which has leaves $\U_j$, for $j = 1,\ldots,\n$, each with a path of length at least $n_0$, (unless we have already successfully made a cycle). We now execute an In-Phase. This involves the construction of trees $T_j,j=1,2,\ldots \nu$. Assume that $P(\U_j)$ is a path from $u_0$ to $v_j$. We start with $\U_j$ and build $T_j$ in a similar way to $T_0$ except that here all paths we generated end with $v_j$. This is done as follows: if a current NPD $\U$ has $P(\U)\in {\cal P}(u,v_i)$ then we consider adding an edge $(w,u)$ and deleting an edge $(w,x)\in \U$. Thus our trees are grown by considering edges directed into the start vertex of each $P(\U)$ rather than directed out of the end vertex. Some technical changes are necessary however.

We consider the construction of our $\nu$ In-Phase trees in two stages. First of all we grow the trees only enforcing condition {\bf C}(ii) of success and thus allow the formation of small cycles and paths. We try to grow the trees until they have a number of leaves in the range $[\n,\a\n\log n]$. The growth of the $\nu$ trees can be considered to be layer by layer as follows: let $L_{j,\ell }$ denote the set of start vertices of the paths associated with the nodes at depth $\ell$ of the $j$-th tree, $j=1,2\ldots ,\nu, \ell = 0,1,\ldots,t_0$. Thus $L_{j,0}=\{ u_0\}$ for all $j$. We prove inductively that $L_{j,\ell}=L_{1,\ell}$ for all $j,\ell$. In fact if $L_{j,\ell}=L_{1,\ell}$ then the acceptable edges have the same set of initial vertices and since all of the deleted edges are edges (enforced by {\bf C} (ii)) we have $L_{j,\ell +1}=L_{1,\ell +1}$. The fact that $L_{j,\ell}=L_{1,\ell}$ gives us some control over the set of new vertices exposed in the In-Phase.

The probability that we succeed in constructing trees $T_1,T_2,\ldots T_{\nu}$ in this way is\\
 $1-O((\log\log n)^2/\log n)$. (This is the probability that we succeed in constructing the $L_{j,\ell}$.)

We now consider the fact that in some of the trees some of the leaves may have been constructed in violation of {\bf C}(i). We imagine that we prune the trees
$T_1,T_2,\ldots T_\nu$ by disallowing any node that was constructed in violation of {\bf C}(i). Let a tree be BAD if after pruning it has less
than $\nu$ leaves and GOOD otherwise. Now an individual pruned tree has been constructed in the same manner as the tree $T_0$ obtained in the
Out-Phase. Thus
$$\Pr(T_1 \mbox{ is BAD})=O\left( {(\log \log n)^2\over \log n}
\right) $$
and
$$\E(\mbox{number of BAD trees}) = O\left( {\nu(\log \log n)^2\over \log n}\right) $$
and
$$\Pr(\exists \geq \nu/2 \mbox{ BAD trees})=O\left({(\log \log n)^2\over \log n}
\right) .$$
Thus
\begin{align*}
&\Pr(\exists <\nu/2 \mbox{ GOOD trees after pruning})\\
&\leq \Pr(\mbox{failure to construct } T_1,T_2,\ldots T_{\nu})+\Pr(\exists \geq \nu/2 \mbox{ BAD trees}) \\
&\leq  O\left(  {(\log \log n)^2\over \log n} \right)
\end{align*}
Thus with probability 1-$O((\log \log n)^2/\log n)$ we end up with $\nu/2$ sets of $\nu$ paths, each of length at least $100n/\log n$ where the $i$'th set of paths all terminate in $v_i$. The sets $in(v_i)$ are still unconditioned and hence
\[
\Pr(\mbox{no $in(v_i)$ edge closes one of these paths}) \leq  \left( 1-{\nu\over n}\right) ^{\nu/2} = o(n^{-1}).
\]
Consequently the probability that we fail to eliminate a particular small cycle $C$ after breaking a particular edge is $O((\log \log n)^2/\log n)$.

If $|C|\geq 4$ then we try once or twice using independent edges of $C$, and so the probability we fail to eliminate  the cycle $C$ is certainly $O(((\log\log n)^2/\log n)^2)$.
To see this, recall that we calculated all probabilities conditional on previous outcomes and assuming $|W|\leq n^{3/4}$. This assumption also holds true for the second attempt to eliminate a small cycle. As there at most $2 \log n$ cycles w.h.p. the probability we fail to eliminate any cycle $C$ with $|C|\geq 4$ after breaking at most two edges is  $O(\log n) \cdot (\log\log n)^4/\log^2 n=o(1)$.

The number of cycles of length 1,2 or 3 in a random cycle cover is asymptotically Poisson with mean 11/6, so w.h.p. there are fewer than $\log\log n$.  Thus the probability we fail to eliminate any  cycle of length at most 3 after breaking a single edge is $O((\log\log n)^3/\log n)=o(1)$.

In summary:
\begin{lemma}
The probability that Phase 2 fails to produce a permutation digraph with minimal cycle length at least $n_0$ is $o(1)$.
\end{lemma}

\subsection{Phase 3. Patching $\cC_i^{*}$  to a Hamilton cycle}
At the end of Phase 2 we have a cycle cover $\cC_i^{*}$ in which the minimum cycle length is $n_0=n/\log n$.
In this section we  transform the cycle cover $\cC_i^{*}$  into a Hamilton cycle, thus proving Theorem \ref{main}. We use a  method particularly suited to very sparse directed graphs,  first employed in \cite{CF1}.

We first show  that w.h.p., the cycles  in $\cC_i^*$ obtained in Phase 2 are not full of small vertices.
\begin{lemma}\label{cyclesmall}
W.h.p., the digraph $D_\bx$ contains no cycle $C$ with $|C|\geq n_0$ and $|C\cap\SM|\geq |C|/100$.
\end{lemma}
\begin{proof}
Let $\e=e^{-c/100k}$ and $n_1=200\e n$. Conditional on the claim in Lemma \ref{smalldegree}, the expected number of relevant cycles can be bounded by
\begin{align*}
&O(n)\sum_{\ell=n_0}^{n_1}\sum_{\substack{\ell_1+\ell_2=\ell\\\ell_1\geq \ell/100}}\binom{\e n}{\ell_1}\binom{n-\e n}{\ell_2}(\ell-1)!\times\\
&\sum_{1\leq d_1,\ldots,d_{\ell_1}\leq c/8k}\;\prod_{r=1}^{\ell_1}\Pr\brac{Z=d_r\bigg|\ Z\leq \frac{c}{8k}}\cdot\frac{d_r}{cn}\cdot
\sum_{1\leq d_1,\ldots,d_{\ell_2}}\;\prod_{r=1}^{\ell_2}\Pr\brac{Z=d_r\ \bigg|\ Z\geq \frac{c}{8k}}\cdot\frac{d_r}{cn}\\
&\leq O(n)\sum_{\ell=n_0}^{n_1} \sum_{\ell_1\geq \ell/100}\frac{(\ell-1)!}{\ell_1!\ell_2!}\cdot \e^{\ell_1}\bfrac{1}{8k}^{\ell_1}\cdot
 \bfrac{1}{c}^{\ell_2}\E\brac{Z\ \bigg|\ Z\geq \frac{c}{8k}}^{\ell_2}\\
&\leq O(n)\sum_{\ell=n_0}^{n_1}\frac{1}{\ell}\sum_{\ell_1\geq \ell/100}\binom{\ell}{\ell_1}\frac{\e^{\ell_1}2^{\ell_2}}{(8k)^{\ell_1}}=O\brac n{\sum_{\ell=n_0}^{n_1}\bfrac{3\e^{1/100}}{(8k)^{1/100}}^\ell}=o(1),
\end{align*}
assuming that $e^{c/10^4k}\geq 3$. \\
After choosing the vertices of the cycle and which are in $\SM$ (the vertices of in- or out-degree at most $c/8k$) we use $d_r/cn$ as a bound on the probability of the existence of a edge in the cycle, and  $2c$ as an upper bound on the conditional expected degree in the last line.
\end{proof}
Fix $i$ and let $C_1,C_2,\ldots,C_\ell$ be the cycles of $\cC_i^{*}$, and let
\begin{equation}\label{Vi}
LARGE=V \sm (W\cup\SM)   \quad\text{ and  } \quad V_j=V(C_j)\sm (W\cup\SM),j=1,2,\ldots,\ell.
\end{equation}
If $\ell=1$, then $C_1$ is already a Hamilton cycle. If $\ell \ge 2$, then let  $c_j = |V_j|$ where $ c_1 \leq c_2 \leq \cdots \leq  c_\ell$, and $c_1 \geq 0.99n_0 -n^{3/4}\geq 9n_0/10$.

For each $C_j$ we consider selecting a set of $\k_j = 2 \rdown{\frac{10c_j}{n_0}}+ 1$ vertices   $v\in C_j\setminus W$, and deleting the edge $(v,u)$ in $\cC_i^{*}$. Let $\k = {\sum_{j=1}^\ell} \k_j$ and relabel (temporarily) the broken edges as $(v_s ,u_s), s\in [\k]$ as follows. In each cycle $C_j$ identify the lowest numbered vertex $x_j$ which lost a cycle edge directed out of it. Put $v_1=x_1$ and then go round $C_1$ defining $v_2,v_3,\ldots { v_{\k_1 }} $ in order. Then continuing with $C_2$, let $v_{\k_1+1}=x_2$ and so on. We thus have $\k$   path sections $P_s=u_{\f(s)}P v_s$ from  $u_{\f(s)}$ to $v_s$ in $\cC_i^{*}$, where $r=\f(s)$ is the label of $u_r$ the initial vertex on the path $P_s$ with terminal vertex $v_s$. By construction, the length of the cycle in $\f$ corresponding to $C_j$ is of odd length $\k_j$. We see that $\f$ is an even permutation as all the  cycles  of $\f$ are of odd length, a fact which was used in \cite{CF1} to establish the inequality \eqref{kk3} given below.

Let $E_{4,i},i=1,2,\ldots,k$ be a random partition of $E_4$ into $k$ almost equal size subsets. It is our intention to rejoin these path sections of $\cC_i^{*}$ using the edges of $E_{4,i}$ to make a Hamilton cycle if we can. In which case, this defines a cyclic permutation $\t$ where $\t(a) =b$  if $P_a$ is joined to $P_b$ by $(v_a,u_{\phi (b)})$, and thus $\t \in  H_\k$ the set of cyclic permutations on $[\k]$. We will use the second moment method to show that a suitable $\t$ exists w.h.p.

The permutation labelling is illustrated in  Figure \ref{fig:Paths}.
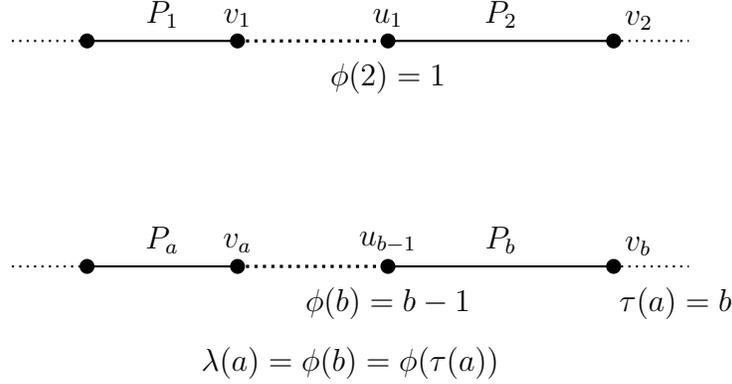
\begin{figure}[H]
\begin{center}
\begin{tikzpicture}[scale=1,inner sep=2pt]
\tikzstyle{vertex}=[circle, fill, inner sep=2pt]
\node  [vertex,label=above:] at (0,3) {};
\node  [vertex ,label=above:$v_1$] at (2,3) {} ;
\node  [vertex,label=above :$u_1$] at (4,3)  {};
\node  [vertex, label=above right:$v_2$] at (7,3)  {};
\draw  [thick] (0,3) -- (2,3) (4,3) -- (7,3);
\draw [very thick, dotted] (2,3)-- (4,3);
\draw [thick, dotted] (-1,3)-- (0,3);
\draw [thick, dotted] (7,3)-- (8,3);
\node[above] at (1,3.1) {$P_1$};
\node[above] at (5.5,3.1) {$P_2$};
\node[below]at (4,2.8) {$\phi(2)=1$};
%
%
\node  [vertex,label=above:] at (0,0) {};
\node  [vertex ,label=above:$v_a$] at (2,0) {} ;
\node  [vertex,label=above :$u_{b-1}$] at (4,0)  {};
\node  [vertex, label=above right:$v_b$] at (7,0)  {};
\draw  [thick] (0,0) -- (2,0) (4,0) -- (7,0);
\draw [very thick, dotted] (2,0)-- (4,0);
\draw [thick, dotted] (-1,0)-- (0,0);
\draw [thick, dotted] (7,0)-- (8,0);
\node[above] at (1,0.1) {$P_a$};
\node[above] at (5.5,0.1) {$P_b$};
\node[below]at (4,-0.2) {$\phi(b)={b-1}$};
\node[below right]at (7,-0.2) { $\tau(a)=b$};
\node[below]at (3.5,-1.0) { $\lambda(a)=\phi(b)=\phi(\tau(a))$};

\end{tikzpicture}
\end{center}
\caption{Original and rearranged path sections showing the permutations  $\phi,\tau,\lambda$}
\label{fig:Paths}
\end{figure}

We will restrict our choices for  $\t$, to produce a variance reduction in the second moment calculation.
In our analysis we  limit our attention to $\t\in R_{\f} =\{ \t \in H_{\k} : \f \t \in H_{\k} \}$.
 Given $\t$, define $\lambda=\f\t$.
If $\t\in R_{\f}$ then we have not only constructed a Hamilton cycle in $\cC_i^{*}\cup E_{4,i}$, but also in the {\em auxiliary digraph} $\Lambda$, whose edges are $(a, \lambda(a))$, $a \in [\k]$,where  $\lambda(a) =\phi(b)$ encodes the fact we joined $P_a$ to $P_b$ using the edge $(v_a,u_{\phi(b)})$.

The following inequality is proved in Lemma 5 of \cite{CF1}:
\beq{kk3}{
(\k-2)! \leq |R_{\f}| \leq (\k-1)!.
}

\begin{figure}[H]
\begin{center}
\begin{tikzpicture}[scale=0.75,inner sep=2pt]
\tikzstyle{vertex}=[circle, fill, inner sep=2pt]
\node  [vertex,label=above: $v_1$] at (1.5,5.2) {};
\node  [vertex,label=above: $u_1$] at (4.5,5.2) {};
\node  [vertex,label=left: $u_3$] at (0,2.6) {};
\node  [vertex,label=right: $v_2$] at (6,2.6) {};
\node  [vertex,label=below: $u_2$] at (4.5,0) {};
\node  [vertex,label=below: $v_3$] at (1.5,0) {};

\draw[very thick, dotted] (1.5,0)--(0,2.6) (1.5,5.2)--(4.5,5.2) (4.5,0)--(6,2.6);
\draw[very thick] (1.5,5.2)--(0,2.6) (1.5,0)--(4.5,0) (4.5,5.2)--(6,2.6);
\node[above] at (0.5,4) {$P_1$};
\node[above] at (5.5,4) {$P_2$};
\node[above] at (3,0.1) {$P_3$};

\begin{scope}[shift={(10,0)}]
\node  [vertex,label=above: $v_1$] at (1.5,5.2) {};
\node  [vertex,label=above: $u_2$] at (4.5,5.2) {};
\node  [vertex,label=left: $u_3$] at (0,2.6) {};
\node  [vertex,label=right: $v_3$] at (6,2.6) {};
\node  [vertex,label=below: $u_1$] at (4.5,0) {};
\node  [vertex,label=below: $v_2$] at (1.5,0) {};

\draw[very thick, dotted] (1.5,0)--(0,2.6) (1.5,5.2)--(4.5,5.2) (4.5,0)--(6,2.6);
\draw[very thick] (1.5,5.2)--(0,2.6) (1.5,0)--(4.5,0) (4.5,5.2)--(6,2.6);
\node[above] at (0.5,4) {$P_1$};
\node[above] at (5.5,4) {$P_3$};
\node[above] at (3,0.1) {$P_2$};

\end{scope}

\end{tikzpicture}
\end{center}
\caption{Left: Original cycle $(P_1,P_2,P_3)$  with permutation $\phi=(1,3,2)$.  Broken edges shown dotted. Right: New cycle $(P_1,P_3,P_2)$ with associated permutations $\tau=(1,3,2)$ and $\lambda=(1,2,3)$. Permutations  given in cycle notation.}
\label{fig:cycles}
\end{figure}
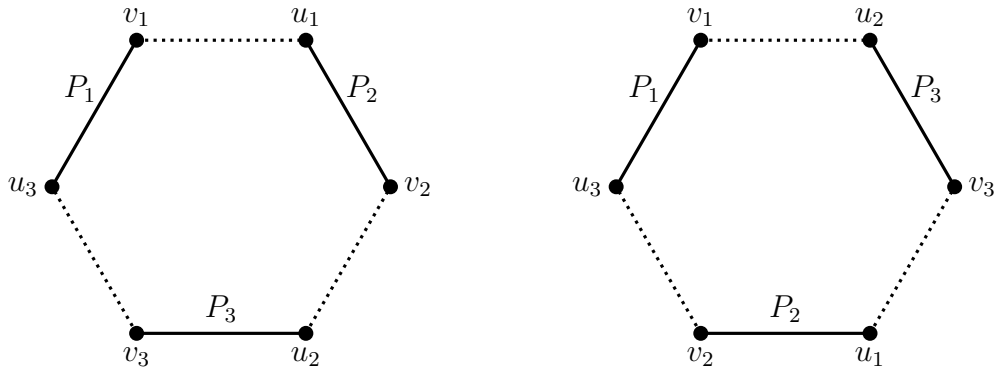

Let $H$ stand for the union of the permutation digraph $\cC_i^{*}$ and $E_{4,i}$. We finish our proof by proving
\begin{lemma}\label{Hcy}
$\Pr(H\text{ does not contain a Hamilton cycle}) =o(1)$.
\end{lemma}
{\em Proof.}
Let $X$ be the number of Hamilton cycles in $H$ obtainable by deleting edges as above,  rearranging the path sections generated by $\f$ according to those $\t \in R_{\f}$ and if possible reconnecting all the sections using edges of $E_{4,i}$. We will use the inequality
\begin{equation}
\label{eq?}
\Pr(X>0)\geq {\E(X)^2\over \E(X^2)}.
\end{equation}
Probabilities in (\ref{eq?}) are with respect to the space of edges $E_{4,i}$ incident with vertices in LARGE, see \eqref{Vi}.
Now the definition of $\k_j = 2 \rdown{{10c_j}/{n_0}}+ 1$  yields that $\k_j\geq 19$, and
\[
\frac{20n}{n_0}-\ell\leq  \k \leq \frac{20n}{n_0}+\ell
\]
and so
$$19\log n\leq \k\leq 21\log n.$$
As $c_j \ge 9 n_0/10$, then
\[
\ell\leq \frac{10\log n}{9}\quad\text{ and also }\quad {c_j\over \k_j}\geq {n_0\over 21},\qquad 1\leq j \leq \ell.
\]
Let $\Omega$ denote the set of possible cycle re-arrangements arising from $\t \in R_\f$. Say $\omega(\t)\in\Omega$ is a {\em success} if $E_{4,i}$ contains the edges $(a,b)$  needed to make a  Hamilton cycle from the path sections arranged according to the cyclic permutation $\t$  where $\t(a) =b$. Thus,
\begin{align}
\E(X) & =  \sum_{\omega\in\Omega}\Pr(\omega\mbox{ is a success}) \nonumber \\
& =  \sum_{\omega\in\Omega} \left(1-\left(1-{1\over n}\right)^{c/8k}\right)^\k  \nonumber\\
& \geq  (1-o(1))\left({c\over 8kn}\right)^\k(\k-2)!\prod_{j=1}^\ell {c_j \choose \k_j} \nonumber  \\
& \geq  \frac{1}{\k \sqrt{\k}} \left( \frac{c\k}{8ken} \right)^\k
\prod_{j=1}^\ell \left(\left( \frac{c_je^{1-1/12\k_j}}{\k_j^{1+(1/2\k_j)}} \right)^{\k_j}\left({1-2\k_i^2/c_i\over \sqrt{2\pi} }\right)\right) \nonumber \\
& \geq  \frac{(2\pi)^{-\ell/2}e^{-\ell/12}}{\k \sqrt{\k}} \left( \frac{c\k}{8ken} \right)^\k
\prod_{j=1}^\ell \left( \frac{2c_j}{\k_j} \right)^{\k_j} \nonumber \\
& \geq \frac{1}{n^{1/8}\k \sqrt{\k}}\left( \frac{c\k}{8ken}  \right)^\k \bfrac{n_0}{11}^\k \nonumber \\
& \to\infty\label{noo2},
\end{align}
for $c$ sufficiently large. We used $\ell \le (10/9) \log n$, and $(10/9)(1/12+ \log \sqrt{2 \pi})<1/8$
to obtain the penultimate line.

Let $M ,M^{\prime}$ be two sets of selected edges which have been deleted from $\cC_i^{*}$ and whose path sections have been rearranged into Hamilton cycles according to $\t, \t^{\prime}$ respectively. Let $N,N'$ be the corresponding sets of edges which have been added to make
the Hamilton cycles. What is the interaction between these two Hamilton cycles?

Let $s=|M\cap M'|$ and $t=|N\cap N'|$, where $t>0$. Now $t\leq s$ since if $(v,u)\in N\cap N'$ then there must be a unique $(\tilde{v},u)=(v_i,u_i) \in M\cap M'$ which is the unique broken $\cC_i^{*}$-edge into $u$ (resp. $(v, \tilde u)=(v_j,u_j)$ which is the unique broken $\cC_i^*$ edge out of $v$). If $i=j$ we are merely replacing the edge we  broke earlier, and we regard these edges as excluded from $M,N$. As we broke at least two cycles, not every edge can have $i=j$. Henceforth we assume that $i,j$ are distinct, which  excludes the case $s=1$.

We claim that $t=s$ implies $t=s=\k$ and $(M,\t )=(M',\t ')$. (This is why we have restricted our attention
to $\t \in R_{\f}$.) Suppose then that $t=s$ and $(v_i,u_i)\in M\cap M'$. Now the edge $(v_i,u_{\lambda (i)})\in N$ and since $t=s$ this edge must also be in $N'$. But this implies that $(v_{\lambda (i)},u_{\lambda (i)})\in M'$ and hence in $M\cap M'$. Repeating the argument we see that $(v_{\lambda ^\ell(i)},u_{\lambda ^\ell(i)})\in M\cap M'$ for all $\ell\geq 0$. But $\lambda$ is cyclic and so our claim follows.

We adopt the following notation. Let $<s,t>$ denote $|M\cap M'|=s$ and $|N\cap N'|=t$ where $t<s$. So
\begin{eqnarray}
\E(X^2) &\leq & \E(X) + (1+o(1))\sum_{M\in\Omega} \left( \frac{c}{8kn}\right) ^{\k} \sum_{\Omega \atop N'\cap N=\emptyset}\left( \frac{c}{8kn} \right) ^{\k} \nonumber \\
& & +(1+o(1))\sum_{M\in \Omega} \left( \frac{c}{8kn} \right) ^{\k} \sum_{s=2}^{\k}
\sum_{t=1}^{s-1}\sum_{\Omega \atop <s,t>} \left( \frac{c}{8kn} \right) ^{\k-t} \nonumber \\
&=&\E(X)+E_1 + E_2 \; \mbox{ say}. \label{eq???}
\end{eqnarray}
Clearly
\begin{equation}
\label{eq????}
E_1 \leq (1+o(1))\E(X)^2.
\end{equation}
For given $\t$, how many $\t^{\prime}$ satisfy the condition $<s,t>$? Previously $|R_{\f}| \geq (\k-2)!$ and now given $<s,t>$, $|R_{\f}(s,t)| \leq (\k-t-1)!$, (consider fixing $t$ edges of $\Lambda^{\prime}$).
\\
Thus
\[
E_2 \leq \E(X)^2 \;
\sum_{s=2}^{\k} \sum_{t=1}^{s-1} {s\choose t}\left[\sum_{\scriptstyle \sigma_1+ \cdots +\sigma_k =s} \prod_{j=1}^\ell \frac{{\k_j \choose \sigma_j}{c_j - \k_j \choose \k_j - \sigma_j}}{ {c_j \choose \k_j}}\right] \frac{(\k-t-1)!}{(\k-2)!} \left( \frac{8kn}{c} \right) ^t .
\]
Now
\begin{eqnarray*}
\frac{{c_j-\k_j \choose \k_j-\sigma_j}}{{c_j \choose \k_j}} & \leq & {{c_j\choose \k_j-\sigma_j}\over {c_j\choose \k_j}} \\
& \leq & ( 1+o(1))\left( \frac{\k_j}{c_j} \right) ^{\s _j}\exp\left\{-{\sigma_j(\sigma_j-1)\over 2\k_j}\right\} \\
& \leq & ( 1+o(1))\left( \frac{22}{n_0} \right) ^{\s _j}\exp\left\{-{\sigma_j(\sigma_j-1)\over 2\k_j}\right\}
\end{eqnarray*}
where the $o(1)$ term is $O((\log n)^3/n)$.
Also
$$\sum_{j=1}^\ell{\sigma_j^2\over 2\k_j}\geq {s^2\over 2\k}\hspace{.25in}\mbox{for }\sigma_1+\cdots \sigma_k=s,$$
$$\sum_{j=1}^\ell{\sigma_j\over 2\k_j}\leq {\ell\over 2},$$
and
$$\sum_{\scriptstyle \sigma_1+ \cdots +\sigma_\ell =s} \prod_{j=1}^\ell
{\k_j \choose \sigma_j}={\k\choose s}.$$
Hence
\begin{align*}
\frac{E_2}{E(X)^2} & \leq (1+o(1))e^{\ell/2}\sum_{s=2}^{\k} \sum_{t=1}^{s-1}{s\choose t}\exp\left\{-{s^2\over 2\k}\right\}\left( {22\over n_0} \right) ^s {\k\choose s}\frac{(\k-t-1)!}{(\k-2)!}  \left( \frac{8kn}{c}\right) ^t  \\
& \leq  (1+o(1))n^{5/9}\sum_{s=2}^{\k} \sum_{t=1}^{s-1} {s\choose t}\exp\left\{-{s^2\over 2\k}\right\}
\left( {22\over n_0} \right) ^s{\k^{s-(t-1)}\over (s-1)!}\left( \frac{8kn}{c}\right)^t  \\
& =  (1+o(1))n^{5/9}\sum_{s=2}^{\k} \left( {22\over n_0} \right) ^s{\k^s\over s!}\exp\left\{-{s^2\over 2\k}\right\}\k\sum_{t=1}^{s-1} {s\choose t}\left( \frac{8kn}{c\k}\right) ^t  \\
& \leq  \frac{c\k^3}{4kn^{4/9}}\;\sum_{s=2}^{\k} \left( {22\over n_0} \right) ^s{\k^s\over s!}\exp\left\{-{s^2\over 2\k}\right\}\bfrac{8kn}{c\k}^{s}\\
& \leq  {c\k^3\over 4kn^{4/9}}\;\sum_{s=2}^{\k}\left({176kn\over c n_0}
\exp\{-{s \over 2\k}\}\right)^s {1\over s!}\\
 = & O(1) \frac{\log^3 n}{n^{4/9}}\;e^{(176 k/c) \log n}=o(1),
\end{align*}
provided $c$ is a sufficiently large constant.

\end{document}